\documentclass{article}%
\usepackage{amsmath}
\usepackage{amsfonts}
\usepackage{amssymb}
\usepackage{url}
\usepackage[boxruled]{algorithm2e}

\usepackage{graphicx}%
\newtheorem{theorem}{Theorem}

\newtheorem{corollary}[theorem]{Corollary}

\newtheorem{example}[theorem]{Example}

\newtheorem{lemma}[theorem]{Lemma}

\newtheorem{proposition}[theorem]{Proposition}
\newtheorem{remark}[theorem]{Remark}

\newenvironment{proof}[1][Proof]{\noindent\textbf{#1.} }{\ \rule{0.5em}{0.5em}}
\begin{document}

\title{Invariant cones for semigroups and controllability of bilinear control systems}
\author{Emerson V. Castelani and Jo\~{a}o A. N. Cossich \and Alexandre J. Santana and Eduardo C. Viscovini\\Departamento de Matem\'{a}tica, Universidade Estadual de Maring\'{a}\\Maring\'{a}, Brazil}
\maketitle

\maketitle
\begin{abstract}
	In this paper we present necessary and sufficient conditions to guarantee the existence of invariant cones, for semigroup actions, in the space of the  $k$-fold exterior product. As  consequence we establish a necessary and sufficient condition  for controllability of a class of  bilinear control systems.
\end{abstract}

\noindent {}\noindent \textit{AMS 2020 subject classification}: 22E46, 93B05, 20M20 

\section{Introduction}
In this paper we deal with invariant cones for semigroup actions and  applications to study controllability of control systems. In our context this question is related with the flag type of the semigroup (in particular semigroup of the control system) and hence with the control sets of the semigroup (or of the control system). Note that it  is far from achieving global results on controllability of bilinear control systems, that is, to find sufficient conditions for controllability is a long term and still incomplete area of research (see e.g. Elliot \cite{Ell}). But, in the last few decades, several papers have been published showing that the Lie theory, especially the theory of semigroups of semisimple Lie groups, provides important tools to study controllability (see e.g. Do Rocio, San Martin and Santana \cite{RSS}, Do Rocio, Santana and Verdi \cite{RoSaVe}, Dos Santos and San Martin \cite{smals} and San Martin \cite{sm}). The semigroup  appears naturally in the context of control systems, in fact, given a bilinear control system 
\begin{equation}
\dot{x}=Ax+uBx,x\in \mathbb{R}^{d}\setminus \{0\},u\in \mathbb{R},
\label{forbilinear}
\end{equation}
where $A$ and $B$ are $d\times d$-matrices, we have that the semigroup $S$ of the system is given by the concatenations of  solutions:
\[S=\{e^{t_{k}(A+u_{k}B)}e^{t_{k-1}(A+u_{k-1}B)} \ldots e^{t_{1}(A+u_{1}B)}, t_{i}\geq 0 , k \in {\Bbb N}  \}   \]
and the group system has a similar definition just changing the positive times  $t_{i}$ by real times  (see e.g Colonius and Kliemann \cite{CoKli} and Elliot \cite{Ell}). And if we consider $A$ and $B$ generating a semisimple Lie algebra $\mathfrak{g}$ we have the possibility to use the semisimple Lie theory to study controllability of the system, for example  in case of  $\mathfrak{g}= \mathfrak{sl}(d,\mathbb{R})$ we have that  this system is controllable in ${\Bbb R}^{d} \setminus \{ 0 \}$ ($Sx= {\Bbb R}^{d} \setminus \{ 0 \}$ for all $x \in {\Bbb R}^{d} \setminus \{ 0 \}$) if and only if $S={\rm Sl}(d,\mathbb{R})$ (see \cite{RSS} and \cite{smt}). 

One of the most interesting ways to prove that the above system is not  controllable is to show the existence of  some $S$-invariant proper subset of ${\Bbb R}^{d}$, a trap of the system. This problem was addressed in \cite{YSac}, by Sachkov,  but in \cite{RSS} the authors  searched these invariant
sets among the convex  cones, since if a set $C$ is invariant by the system then the convex closure of $C$ is also invariant. In our work we follow similar approach to improve and generalize the results contained in \cite{RSS} and in particular to give a necessary and sufficient condition for controllability of the above system when $A, B \in {\frak sl}(d,{\Bbb R})$. More specifically, we prove that the system is controllable if and only if it does not have an invariant proper cone in the $k$-fold exterior product of ${\Bbb R}^{d}$, $\bigwedge ^{k}{\Bbb R}^{d}$, for all $k \in \{1, \ldots , d-1 \}$. In fact, this  is a consequence of our following transitivity result: 	Let $S\subset {\rm Sl}(d,\mathbb{R})$ be a connected semigroup with nonempty interior. Then $S={\rm Sl}(d,\mathbb{R})$ if and only if there are no   $S$-invariant and proper cones in  $\bigwedge^k\mathbb{R}^d$, for all $k\in\{1,\cdots,d-1\}$. These two results are   built from the theory of   flag type of a semigroup.

We briefly recall the main concept or tool of this paper.  Consider $S\subset \mathrm{Sl}(d,\mathbb{R})$  a
semigroup with nonempty interior. Denote by  $\mathbb{F}_{\Theta }$ the flag manifold of all flags 
$ (V_{1}\subset \cdots \subset V_{k})  $
of subspaces $V_{i}\subset \mathbb{R}^{d}$ with $\dim V_{i}=r_{i}$, 
$i=1,\ldots ,k$ and $\Theta = \{r_{1}, \ldots , r_{k} \}$.  Take the canonical projection   $\pi _{\Theta _{1}}^{\Theta }:\mathbb{F}_{\Theta }\rightarrow \mathbb{F}%
_{\Theta _{1}}$  with ${\Theta _{1}} \subset {\Theta}$ and denote by $\mathbb{F}$ the full flag manifold with the sequence ${\Theta}_{M} =\{1,2,\ldots ,d-1\}$. There is a natural (transitive) action of $\mathrm{Sl}(d,\mathbb{R})$ in these flag manifolds, then  an invariant control set, in $\mathbb{F}_{\Theta }$, for the $S$-action   is
a subset $C \subset \mathbb{F}_{\Theta }$ such that $\mathrm{cl}\left(
Sx\right) =C$,  for all $x\in C$,  and $C$ is maximal with this property. Recall that an invariant control set is closed and its interior is dense on it. One important result is that in each flag manifold $\mathbb{F}_{\Theta }$
there exists just one $S$-invariant control set. Moreover, there exist $\Theta \subset {\Theta}_{M}$ such that ${\pi_{\Theta}^{-1}}(C_{\Theta})=C$ where ${\pi}_{\Theta}: \mathbb{F} \rightarrow {\mathbb{F}}_{\Theta}$ is the canonical projection, and $C_{\Theta}, C$ are the invariant control sets in $\mathbb{F} , {\mathbb{F}}_{\Theta}$ respectively. In addition, among 
 these flag manifolds there
is exactly one, denoted  by $\mathbb{F}_{\Theta \left( S\right) }$, which is
minimal (see \cite{sm}). The flag manifold $\mathbb{F}_{\Theta \left( S\right) }$  (or $\Theta \left( S\right) $) is called the  flag (or parabolic) type of $S$ (for details see San Martin \cite{flagType} and San Martin and Tonelli \cite{smt}). We note that once we know the invariant control set $C_{\Theta \left(
	S\right) }$ in the flag type $\mathbb{F}_{\Theta \left( S\right) }$
then every invariant control set is described because for any $\Theta $ we have 
$C_{\Theta }=\pi _{\Theta }\left( C\right) $ and $C=\pi _{\Theta \left(
	S\right) }^{-1}\left( C_{\Theta \left( S\right) }\right) $.
 Given $\Theta =\{r_{1},\ldots
,r_{n}\}$ with $0<r_{1} < \cdots <  r_{n}<d$ define $\Theta
^{*}=\{d-r_{n},\ldots ,d-r_{1}\}$. The flag manifold $\mathbb{F}_{\Theta^{*}}$ is said to be dual of $\mathbb{F}_{\Theta }$. With this we have that the flag type of $S^{-1}$ is given by the flag manifold $\mathbb{F}_{\Theta \left( S\right)^{*}}$ dual to the flag type of $S$ (see \cite{smax}).

From this  semigroup theoretical development, considering $S$ a connected semigroup with nonempty interior and taking $\Theta (S)$ its flag type,  we prove our main result: there exists a non-trivial $S$-invariant cone $W \subset \bigwedge ^{k}{\Bbb R}^{d}$
if and only if  $k \in \Theta (S)$. Hence, as a consequence we show the controllability and transitivity results mentioned above.

About the structure of this paper, in  Section 2 we establish the main concepts necessary for the next sections. In the third section we study the invariance of cones in  $\bigwedge ^{k}{\Bbb R}^{d}$. In Section 4 we have the main results of this paper, in special we present necessary and sufficient conditions for the existence of these invariant cones. In Section 5 we extend the previous results to convex set instead of cones. And finally, in the last section we present some examples where the fundamental conditions are verified using a computational implementation created by the authors. 

\section{Preliminaries}

In this section we collect some concepts and facts about flag manifolds related with the $\mathrm{Sl}(d,\mathbb{R})$ actions. In particular we are interested  in some  dynamical properties originating from the action of semigroups  $S \subset \mathrm{Sl}(d,\mathbb{R})$ on flag manifolds (for details see San Martin \cite{SM} and \cite{SM2}).

Now we complement the previous introduction about flag manifolds and control sets. Recall that the flag manifolds $\mathbb{F}_{\Theta}$ are compact and the \textbf{minimal flag manifolds} are the Grassmannians $\mathbb{F}_\Theta=\mathbb{G}_{k}(d)$, where $\Theta=\{k\}$. A particular case, when $k=1$, is the projective space ${\Bbb P}^{d-1}=\mathbb{G}_1(d)$.

From now on, in this section we discuss the special case $\mathbb{G}_k(d)$, $1\leq k\leq d-1$. In this work  it is convenient represent    $\mathbb{G}_k(d)$ in the following algebraic way. Let $B_k(d)$ be the set of $d\times k$ matrices of rank $k$. Define in $B_k(d)$ the following equivalence relation: $p\sim q$ if exists $a\in {\rm Gl}(k,\mathbb{R})$ with $q=pa$. In other words, $p\sim q$ if, and only if, the columns of $p$ and $q$ generate the same subspace of $\mathbb{R}^d$. Then we can see $\mathbb{G}_k(d)$ as $B_k(d)/\sim$. Denote the elements of $\mathbb{G}_k(d)$ by $[p]$. There is a natural action $\rho_k$ of the Lie group ${\rm Sl}(d,\mathbb{R})$ on $\mathbb{G}_k(d)$, which is given by $\rho_k(g,[p])=[gp]$.

Now take an arbitrary basis $\mathcal{B}$ of $\mathbb{R}^d$ and 
$N_\mathcal{B}$ the nilpotent group of lower triangular matrices (with respect to $\mathcal{B}$) with ones on the main diagonal. The decomposition of  $\mathbb{G}_k(d)$ into $N_\mathcal{B}$-orbits is called   \textbf{Bruhat decomposition}   of $\mathbb{G}_k(d)$,  moreover if we change the basis the decomposition also changes. There is just a finite number of these orbits,   $N_\mathcal{B}[p]$ with $[p]\in\mathbb{G}_k(d)$. It is well known that exists only one open and dense orbit, $N_\mathcal{B}[p_0]$, where $[p_{0}]$ is the subspace spanned by the first $k$ basic vectors (see \cite{SM}). We have that $N_\mathcal{B}[p_0]$ can be written as 
$$\left[
\begin{array}{c}
I_k \\
X
\end{array}
\right]$$
with $I_k$  the $k\times k$ identity and $X$  an arbitrary $(d-k)\times k$ matrix. Taking   
$$\eta= \left[
\begin{array}{cc}
A_1 & 0 \\
Y & A_2
\end{array}
\right] \in N_{\mathcal{B}}$$
with $A_1$ and $A_2$ invertible, it follows that
$$\rho_k(\eta,[p_{0}])=[\eta p_0]= \left[
\begin{array}{cc}
A_1 & 0 \\
Y & A_2
\end{array}
\right]
\left[
\begin{array}{c}
I_k \\
0
\end{array}
\right]=
\left[
\begin{array}{c}
A_1 \\
Y
\end{array}
\right]=
\left[
\begin{array}{c}
I_k \\
YA_1^{-1}
\end{array}
\right] .
$$
Note  that this orbit is  diffeomorphic to euclidean spaces.

Another important concept here is the \textbf{split regular} or just \textbf{regular}  element, that is the $h \in {\rm Sl}(d,\mathbb{R})$ with positive and distinct eigenvalues, where in some basis (denoted by $\mathcal{B}(h)$), $h={\rm diag} \{ \lambda_{1} , \ldots , \lambda_{d} \}$ with $\lambda_{1} > \cdots>\lambda_{d} > 0$.
Considering the action on $\mathbb{G}_k(d)$, the  fixed points for $h$   are the subspaces spanned by $k$ basic vectors. Moreover, these fixed points are hyperbolic and
 with respect to $\mathcal{B}(h)$, the stable manifolds are the
$N_{\mathcal{B}}$-orbits.
 One interesting dynamical property is that  the stable manifold of the subspace
$[p_{0}]$ is open and dense so that it is the unique 
attractor  for $h$  and $h^m[q]\rightarrow[p_{0}]$  for generic $[q]$. Now taking  $h^{-1}$  instead of $h$ and reverting the order of the basis, it follows that $h$  has also just one repeller, and it is  the subspace spanned by
the last $k$ basic vectors $\{ e_{d-k+1}, \ldots, e_{d} \}$.

We recall other  dynamical facts. Let $S \subset {\rm Sl}(d,\mathbb{R})$ be a semigroup with nonempty interior and denote by $\text{reg}(S)$ the set of regular elements in ${\rm int}S$. As before take $C_{k}$  the $S$-invariant control set  in  $\mathbb{G}_k(d)$, its  uniqueness implies that
$$C_{k}=\displaystyle\bigcap_{[p]\in \mathbb{G}_k(d)} {\rm cl}(S[p]).$$ According to the above comments, for $h\in\text{reg}(S)$ we have that $b_{\{k\}}(h)=[p_0]$ and $C_k\subset N_{\mathcal{B}(h)}[p_0]$ if $k\in\Theta(S)$. The  \textbf{set of transitivity} of an invariant control set  $C_{k}$ is  the set $C_{k}^0$ of the fixed points which are the attractors for  elements in $\text{reg}(S)$ (see \cite{sm}). Specifically,  we have that for any $[p] \in C_{k}^0$, there exists a basis $\mathcal{B}(h)=\{e_1,\ldots,e_d\}$ of $\mathbb{R}^{d}$ and $h={\rm diag} \{ \lambda_{1} , \ldots , \lambda_{d} \}$ with $\lambda_{1} > \cdots> \lambda_{d} > 0$ (in this basis), such that $h \in {\rm int}S$ 
and $[p]=\langle e_1 , \ldots , e_k\rangle$, i.e.,  $[p]$ is the attractor of $h$. From this fact it 
follows that the set of attractors of elements in  $\text{reg}(S)$ coincides with  $C_k^0$ and this set is dense in $C_k$. Hence $\text{reg}(S)$ is dense in ${\rm int}S$ and  $C_k$ is formed, in some sense, by   attractors for these regular elements. This is a kind of converse to the fact that  $[p]\in C_k$ if $[p]$ is the attractor of
a element  $h \in  \text{reg}(S)$. Therefore $C_{k}$ is contained in the open Bruhat component corresponding to $\mathcal{B}(h)$.
Another interesting result in this context is that $C_{k} =
 \mathbb{G}_k(d)$ for some $k$ if and only if $S$ is transitive on $\mathbb{G}_k(d)$. On the other hand, we have that if $S$ is a proper semigroup of ${\rm Sl}(d,\mathbb{R})$, then $C_{k} \neq  \mathbb{G}_k(d) $ for any  $k \in \{ 1, \ldots , d-1 \}$ and $S$ is not transitive on $\mathbb{G}_k(d) $ (see \cite{smt}). 
 
We finish this section recalling some necessary facts about tensorial product and Grassmanianns.

For $k\in\{1,\ldots,d\}$, denote by $\bigwedge^k\mathbb{R}^d$ the $k$-fold exterior product of $\mathbb{R}^d$ and let $\mathcal{F}_k(d)$ be the set of all $k$ multi-index $I=\{i_1,\ldots,i_k\}\subset\{1,\ldots,d\}$ with $1\leq i_1<\cdots<i_k\leq d$. It is well known that if we fix a basis $\mathcal{B}=\{e_1,\ldots,e_d\}$, then $\{e_I:=e_{i_1}\wedge\cdots\wedge e_{i_k}; I=\{i_1,\ldots,i_k\}\in\mathcal{F}_k(d)\}$ is a basis of $\bigwedge^k\mathbb{R}^d$. Along the text, we use the notation $\mathcal{D}$ to designate the set of all decomposable elements of $\bigwedge^k\mathbb{R}^d$, that is, the set of elements that can be written as $u_1\wedge\cdots\wedge u_k$ with $u_i \in {\mathbb{R}}^{d}$.

The manifold $\mathbb{G}_k(d)$, $k\in\{1,\ldots,d-1\}$, can be seen as a compact submanifold of the projective space $\mathbb{P}\left(\bigwedge^k\mathbb{R}^d\right)$ of $\bigwedge^k\mathbb{R}^d$ via  Pl\"ucker embedding $\varphi:\mathbb{G}_k(d)\rightarrow\mathbb{P}\left(\bigwedge^k\mathbb{R}^d\right)$, $\varphi([p])=[u_1\wedge\cdots\wedge u_k]$, where $p=[u_1 \ \ldots \ u_k]$ is a $d\times k$ matrix and $[u_1\wedge\cdots\wedge u_k]\in\mathbb{P}\left(\bigwedge^k\mathbb{R}^d\right)$ denotes the class of all non-zero multiples of $u_1\wedge\cdots\wedge u_k\in\bigwedge^k\mathbb{R}^d$.

Identifying the Grassmaniann $\mathbb{G}_k(d)$ as a subset of $\mathbb{P}\left(\bigwedge^k\mathbb{R}^d\right)$, we can write the action $\rho_k$ of ${\rm Sl}(d,\mathbb{R})$ on $\mathbb{G}_k(d)$ as 
$$\rho_k(g,[u_1\wedge\cdots\wedge u_k])=[gu_1\wedge\cdots\wedge gu_k]$$
and  denote $\rho_k(g,[p])$ simply by $g[p]$ .

In the next sections $\pi:\left(\bigwedge^k\mathbb{R}^d\right)\backslash\{0\}\rightarrow\mathbb{P}\left(\bigwedge^k\mathbb{R}^d\right)$ represents the canonical projection.

\section{Cones in $k$-fold exterior product}
From now on we consider a  connected semigroup $S\subset\text{Sl}(d,\mathbb{R})$  with nonempty interior. In this work a {\bf cone} means a closed convex cone in a finite dimensional vector space $V$ and if not otherwise specified the cones are proper and non-trivial. Remember that a cone $W$ is {\bf pointed} if $W\cap-W=\{0\}$ and {\bf generating} if $\text{int}W\neq\emptyset$. Our main interest is to study the $S$-invariance of this kind of cones in  $\bigwedge^k\mathbb{R}^d$, with $1\leq k\leq d-1$.

In this section we present some technical and useful results about cones. In particular, we show that  any $S$-invariant   cone $W\subset\bigwedge^k\mathbb{R}^d$ contains a decomposable element and it is  pointed and generating (see the following Propositions \ref{lemdec} and \ref{lem7}). To obtain these results we need some lemmas.
\begin{lemma}
	\label{um}
	Let $F:V_1\rightarrow V_2$ be an analytic map where  $V_1$ and $V_2$ are  finite dimensional vector spaces. Assume that for a nonempty open set $U\subset V_1$ there is a subspace $V\subset V_2$ such that $F(U)\subset V$. Then $F(V_1)\subset V$.
\end{lemma}

\begin{proof}
	The canonical projection $p:V_2\rightarrow V_2/V$ is linear and then analytic. Therefore, $p\circ F$ is an analytic function, and $p\circ F(U)=0+V$, since $p(x)\in V$ for all $x\in U$. Therefore, $p\circ F(x)=0+V$, for every $x\in V_1$, because the unique analytic map between finite dimensional vector spaces which vanishes on an open subset of the domain is the null map. Hence, $F(x)\in V$, for all $x\in V_1$.
\end{proof}

\begin{lemma}
	\label{dois}
	Let $V$ be a $d$-dimensional vector space  and take the cone  $W\subset V$. Then $W$ is generating if, and only if, $W$ is not contained in any proper subspace of $V$.
\end{lemma}

\begin{proof}
	If $W$ has nonempty interior, then $W$ is not  contained in a proper subspace of $V$.
	For the converse, observe that  convex cones spanned by any basis of $V$ have nonempty interior. In fact, let $\{e_1,\ldots,e_d\}$ be a basis of $V$. The convex cone spanned by this basis is the set
	$$\left\{\sum_{i=1}^d\alpha_ie_i;\ \alpha_1,\ldots,\alpha_d\ge0\right\} . $$
	 Then the interior of this set is nonempty. Now, assuming that $W$ is not contained in a proper subspace of $V$, we have that $W$ contains a basis  $\mathcal{B}$. Since $W$ is a convex cone it follows that $W$ also contains the convex cone spanned by $\mathcal{B}$. Therefore, $\text{int}W\neq\emptyset$.
\end{proof}

\begin{lemma}\label{theo4}
	If $U$ is open in the set of decomposable elements $\mathcal{D}$ (in the relative topology), then $U$ contains a basis of $\bigwedge^k\mathbb{R}^d$.
\end{lemma}

\begin{proof}
	First note that we can write the set of decomposable elements as the image of a polynomial function. In fact, let $F:(\mathbb{R}^d)^k\rightarrow\bigwedge^k\mathbb{R}^d$ be the map given by $F(u_1,u_2,\cdots,u_k)=u_1\wedge u_2\wedge\cdots\wedge u_k$. Clearly, $F((\mathbb{R}^d)^k)=\mathcal{D}$. On the other hand, $F$ is polynomial due to the multi-linearity of the wedges, hence it is analytic.
	
	Since $F$ is continuous, the set $F^{-1}(U)$ is open in $(\mathbb{R}^d)^k$. In fact, since $U$ is open in $\mathcal{D}$, there is an open $U^\prime\subset\bigwedge^k\mathbb{R}^d$ with $U=U^\prime\cap\mathcal{D}$. So	$F^{-1}(U)=$ $F^{-1}(U^\prime)\cap(\mathbb{R}^d)^k=$ $F^{-1}(U^\prime)$.
	
	Now, suppose that $U$ does not contain a basis of $\bigwedge^k\mathbb{R}^d$, then $S$ is contained in a proper subspace $Z$, so $F(F^{-1}(U))\subset U\subset Z$, and by Lemma \ref{um}, $F((\mathbb{R}^d)^k)\subset Z$. Hence $\mathcal{D}$ is contained in a proper subspace of $\bigwedge^k\mathbb{R}^d$. This is a contradiction, because $\mathcal{D}$ spans $\bigwedge^k\mathbb{R}^d$. Therefore $U$ contains a basis of $\bigwedge^k\mathbb{R}^d$.
\end{proof}

In the next two propositions we consider the representation $\delta:{\rm Sl}(d,\mathbb{R})\rightarrow {\rm Gl}(\bigwedge^k\mathbb{R}^d)$ where
$$\delta(g)(u_1\wedge\cdots\wedge u_k):=gu_1\wedge\cdots\wedge gu_k.$$
To abbreviate, we  denote $\delta(g)$ simply by $g$.

Now we can prove that an invariant  cone  contains a decomposable element.

\begin{proposition}\label{lemdec}
	Take $S\subset {\rm Sl}(d,\mathbb{R})$ a semigroup with nonempty interior. Let $\{0\}\neq W\subset\bigwedge^k\mathbb{R}^d$ be an $S$-invariant cone. Then $W$ intercepts a non-null decomposable element of $\bigwedge^k\mathbb{R}^d$.
\end{proposition}

\begin{proof}
	Take $h\in \text{reg}(S)$ and consider as before the basis $\mathcal{B}=\{e_1,\cdots,e_d\}$   of $\mathbb{R}^d$ such that   $h=\text{diag}(\lambda_1,\cdots,\lambda_d)$  with $\lambda_1>\cdots>\lambda_d>0$. Note that for $I=\{i_1,\ldots,i_k\}\in\mathcal{F}_k(d)$, the vectors $e_I=e_{i_1}\wedge\cdots\wedge e_{i_k}\in\bigwedge^k\mathbb{R}^d$ are eigenvectors of $h$, with eigenvalues $\lambda_{i_1}\cdots\lambda_{i_k}$. Moreover, they form a basis of $\bigwedge^k\mathbb{R}^d$.
	
	Define the following order relation on $\mathcal{F}_k(d)$: given $I=\{i_1,\ldots,i_k\}$ and $J=\{j_1,\ldots,j_k\}$ in $\mathcal{F}_k(d)$,
	$$I\prec J \mbox{ if } \lambda_{i_1}\cdots\lambda_{i_k}<\lambda_{j_1}\cdots\lambda_{j_k}.$$
	
	If necessary, we take a perturbation of   $h=\text{diag}(\lambda_1,\ldots,\lambda_d)\in\text{int}S$ such that $\prec$ become a total order. Consider $0\neq v\in W$ with $v=\sum_{I\in\mathcal{F}_k(d)}\alpha_{I}e_{I}$ and define 
	$$J_0=\{j_1,\ldots,j_k\}=\max\{I\in\mathcal{F}_k(d);\ \alpha_I\neq0\}.$$
	
	As $W$ is  $S$-invariant,   we have that $W$ is invariant under $h$. Hence, 
	$$\left(\frac{h^m(v)}{(\lambda_{j_1}\cdots\lambda_{j_k})^m}\right)_{m\in\mathbb{N}}$$
	is a sequence in $W$ and
	\[
		\frac{h^m(v)}{(\lambda_{j_1}\cdots\lambda_{j_k})^m}=\sum\limits_{I\in\mathcal{F}_k(d)}\alpha_{I}\frac{h^m(e_{I})}{(\lambda_{j_1}\cdots\lambda_{j_k})^m}
		=\sum\limits_{I\in\mathcal{F}_k(d)}\alpha_{I}\frac{(\lambda_{i_1}\cdots\lambda_{i_k})^m}{(\lambda_{j_1}\cdots\lambda_{j_k})^m}e_I,
	\]
	for all $m\in\mathbb{N}$. Note that if $I=\{i_1,\cdots,i_k\}\notin\{I\in\mathcal{F}_k(d);\ \alpha_I\neq0\}$, then $\alpha_{I}\dfrac{(\lambda_{i_1}\cdots\lambda_{i_k})^m}{(\lambda_{j_1}\cdots\lambda_{j_k})^m}=0$, for all $m\in\mathbb{N}$. Moreover
	$\lambda_{j_1}\cdots\lambda_{j_k}>\lambda_{i_1}\cdots\lambda_{i_k}$ for all $\{i_1,\ldots,i_k\}$ in $\{I\in\mathcal{F}_k(d);\ \alpha_I\neq0\}\backslash\{J_0\}$.
	Hence,
	$$\lim_{m\rightarrow\infty}\dfrac{(\lambda_{i_1}\cdots\lambda_{i_k})^m}{(\lambda_{j_1}\cdots\lambda_{j_k})^m}=0.$$
	
	Therefore
	\begin{eqnarray*}
		\displaystyle\lim_{m\rightarrow\infty}\frac{h^m(v)}{(\lambda_{j_1}\cdots\lambda_{j_k})^m}=\displaystyle\lim_{m\rightarrow\infty}\sum\limits_{I\in\mathcal{F}_k(d)}\alpha_{I}\frac{(\lambda_{i_1}\cdots\lambda_{i_k})^m}{(\lambda_{j_1}\cdots\lambda_{j_k})^m}e_I=\alpha_{J_0}e_{J_0}
	\end{eqnarray*}
	
	The closeness of $W$  implies that the decomposable element $\alpha_{J_0}e_{J_0}$ belongs to $W$, and moreover, this element is non-null.
\end{proof}

Hence we have the main result of this section.

\begin{proposition}\label{lem7}
	Let $S\subset {\rm Sl}(d,\mathbb{R})$ be a semigroup with non empty interior. If $\{0\}\neq W\subset\bigwedge^k\mathbb{R}^d$ is a  $S$-invariant cone, then $W$ is pointed and generating.
\end{proposition}

\begin{proof}
	First recall that the representation of ${\rm Sl}(d,\mathbb{R})$ on $\bigwedge^k\mathbb{R}^d$ is irreducible. 

Now, define $H=W\cap-W$. Then $H$ is an $S$-invariant vector subspace.
We have also that $H$ is $S^{-1}$-invariant, because if $g\in S$, then $gH\subset H$. Since $g$ is invertible, $gH$ is a subspace of $H$ with $\dim gH=\dim H$,
i.e., $gH=H$. Consequently, $H=g^{-1}H$. The fact that $\text{int}S\neq\emptyset$ implies that ${\rm Sl}(d,\mathbb{R})$ is generated by $S\cup S^{-1}$. Hence $H$ is ${\rm Sl}(d,\mathbb{R})$-invariant, now knowing that $W$ is proper and  ${\rm Sl}(d,\mathbb{R})$ is irreducible we have that  $H=\{0\}$. Hence $W$ is pointed.

Finally, assume that $\text{int}W=\emptyset$. By Lemma \ref{dois}, $W\cup-W$ is contained in a proper subspace $V$ of $\bigwedge^k\mathbb{R}^d$. Consider a decomposable element $x\in W$ and take $\rho_{k}^{q}:{\rm Sl}(d,\mathbb{R})\rightarrow\mathbb{G}_k(d)$ the open map $\rho_{k}^{q}(g)=[gq]$  where $[q]:=\varphi^{-1}(\pi(x))$ and $\varphi$ is the Pl\"ucker embedding defined in the second section. Then $\varphi(\rho_{k}^{q}(\text{int}S))$ is open in $\varphi(\mathbb{G}_k(d))$, that is, there exists an open set $B\subset\mathbb{P}\left(\bigwedge^k\mathbb{R}^d\right)$ such that
$$\varphi(\rho_{k}^{q}(\text{int}S))=B\cap\varphi(\mathbb{G}_k(d)).$$

Knowing that
\[
\pi^{-1}(\varphi(\phi(\text{int}S)))=\pi^{-1}(B\cap\varphi(\mathbb{G}_k(d)))=\pi^{-1}(B)\cap\mathcal{D} ,
\]
we have that $\pi^{-1}(\varphi(\phi(\text{int}S)))$ is open in $\mathcal{D}$. By Lemma \ref{theo4}, $\pi^{-1}(\varphi(\phi(\text{int}S)))$ contains a basis of $\bigwedge^k\mathbb{R}^d$. Note also that 
$$\pi^{-1}(\varphi(\phi(\text{int}S)))=\pi^{-1}(\varphi((\text{int}S)[q]))=\pi^{-1}(\pi((\text{int}S)x))=\pi^{-1}((\text{int}S)\pi(x)).$$

So, if $y\in \pi^{-1}(\varphi(\phi(\text{int}S)))$, then $\pi(y)\in(\text{int}S)\pi(x)$, hence there is $g\in\text{int}S$ with $\pi(y)=g\pi(x)=\pi(gx)$, that is, $y=\alpha gx$ for some $\alpha\neq0$. If $\alpha>0$, then $y\in\alpha Sx\subset\alpha W=W$ and if $\alpha<0$, then $y\in\alpha Sx\subset\alpha W=-W$. Anyway $y\in W\cup-W$ and we conclude that $\pi^{-1}(\varphi(\phi(\text{int}S)))\subset W\cup-W$. But it is a contradiction, because $\pi^{-1}(\varphi(\phi(\text{int}S)))$ is contained in the proper subspace $V$ of $\bigwedge^k\mathbb{R}^d$ and contains a basis of $\bigwedge^k\mathbb{R}^d$. Therefore, $\text{int}W\neq\emptyset$.
\end{proof}

\section{Cones, flag type and controllability}

In this section we prove that there exists an $S$-invariant cone in $\bigwedge^k\mathbb{R}^d$ if and only if  the flag type of $S$ contains $k$.  Consequently we have the main result of this section,  Theorem \ref{controllab}, that gives a necessary and sufficient condition for the equality   $S=\text{Sl}(d,\mathbb{R})$ in terms of the existence of $S$-invariant cones in the spaces $\bigwedge^k\mathbb{R}^d$, $k\in\{1,\ldots,d-1\}$. As application, we determine necessary and sufficient conditions for the controllability of a bilinear control system.  

\begin{theorem}\label{theo5}
	Let $S\subset {\rm Sl}(d,\mathbb{R})$ be a connected semigroup with flag type given by $\Theta(S)$. If $k\in\Theta(S)$, then there exists an $S$-invariant cone $\{0\}\neq W\subset\bigwedge^k\mathbb{R}^d$.
\end{theorem}

\begin{proof}
	Take $h\in \text{reg}(S)$   and consider $\mathcal{B}(h)=\{e_1,\ldots,e_d\}$ the special basis of $\mathbb{R}^d$. We saw in the second section that  $b_{\{k\}}(h)=(\mbox{span}\{e_1,\ldots,e_k\})$ and the orbit 
	$$N_{\mathcal{B}(h)}b_{\{k\}}(h)=\left\{\left[
	\begin{array}{c}
	I_k\\
	X\\
	\end{array}
	\right];\ X\in\mathbb{R}^{(d-k)\times k}\right\}$$
	contains $C_{k}$.
	Note that $\varphi(N_{\mathcal{B}(h)}b_{\{k\}}(h))\subset\pi(M)$, where $M$ is the affine subspace
	$$M=\left\{(1,x_2,\cdots,x_{{d}\choose{k}});\ x_2,\cdots,x_{{d}\choose{k}}\in\mathbb{R}\right\}\subset\bigwedge^k\mathbb{R}^d$$
	in the basis $\{e_I;\ I\in\mathcal{F}_k(d)\}$.
	Since the invariant control set $C_k\subset \mathbb{G}_k(d)$ is contained in $N_{\mathcal{B}(h)}b_{\{k\}}(h)$, we have 
	$$\varphi(C_k)\subset\varphi(N_{\mathcal{B}(h)}b_{\{k\}}(h))\subset\pi(M).$$

	Define $M_1:=\pi^{-1}(\varphi(C_k))\cap M$. Let $W$ be the cone generated by $M_1$,  $W$ is clearly non-null.
	
	Now, we  show that $W$ is $S$-invariant. Since $C_k$ is $S$-invariant, it follows that $\varphi(C_k)$ is $S$-invariant. We claim that $(\mathbb{R}\backslash\{0\})M_1$ is $S$-invariant. In fact, given $\alpha\in\mathbb{R}\backslash\{0\}$, $u_1\wedge\cdots\wedge u_k\in M_1$ and $g\in S$, we have that
	$\pi(g(\alpha\ u_1\wedge\cdots\wedge u_k))$ $=\pi(g(u_1\wedge\cdots\wedge u_k))$ is contained  in $\pi(gM_1)=g\pi(M_1)=g\varphi(C_k)\subset\varphi(C_k)$,
due to the equality   $\pi(M_1)=\varphi(C_k)$ and the $S$-invariance of $\varphi(C_k)$. Hence knowing  that $\pi|_M$ is injective, we conclude the claim. As $S$ is  connected this implies that $C_k$, $\varphi(C_k)$ and $M_1$ are connected.
	
	Furthermore, since for every $x\in\left(\bigwedge^k\mathbb{R}^d\right)\backslash\{0\}$ the mapping	$g\in S\mapsto gx\in\bigwedge^k\mathbb{R}^d$ is continuous, we conclude that $S$ leaves invariant the
	connected components of $(\mathbb{R}\backslash\{0\})M_1$. As $(\mathbb{R}^{+})M_1$ is one of these components,  $(\mathbb{R}^{+})M_1$ is invariant, implying that its convex closure $W$	is $S$-invariant.
\end{proof}  

\begin{remark}
Our result generalizes Theorem 4.2 in \cite{RSS} and also improves its hypotheses in the sense that we do not need to have the identity in ${\rm cl}S$. In \cite{RSS} the authors assume $1 \in S$ to guarantee that $S$ leaves invariant the connected components of $({\Bbb R} \backslash \{0\})M_{1}$, but we can show that this is not necessary. In fact, let $g \in S$, then $g$ leaves  $({\Bbb R} \backslash \{0\})M_{1}$ invariant. So $g$ is a bijection between the connected components of $({\Bbb R} \backslash \{0\})M_{1}$. Denote by $M_{1}^{+}=({\Bbb R}^{+})M_{1}$ and  $M_{1}^{-}=({\Bbb R}^{-})M_{1}$ these connected components. Suppose that there is an element $g \in S$ which does not leave  $M_{1}^{+}$ invariant. Then $g(M_{1}^{+})=M_{1}^{-}$ and  $g(M_{1}^{-})=M_{1}^{+}$. Hence we have another element in $S$,   $g^{2}$, that leaves invariant the components, but  this contradicts the connectedness of $S$.  
\end{remark}

 We also note that by Proposition \ref{lem7}, the cone $W$, in the above theorem  is pointed and generating.

The following results prove that the existence of a pointed invariant cone in $\bigwedge^k\mathbb{R}^d$ implies that the flag type of the semigroup contains $k$.

\begin{lemma}
	\label{lema}
	Assume that $k\notin\Theta(S)$. Let $C_k$ be the invariant control set for the action of $S$ on $\mathbb{G}_k(d)$. Then there is a two-dimensional subspace $V\subset\bigwedge^k\mathbb{R}^d$ such that $\pi(V)\subset\varphi(C_k)$.
\end{lemma}
\begin{proof}
	Denote by $\pi_k:\mathbb{F}\rightarrow\mathbb{G}_k(d)$ the natural projection and consider $[p]\in C_k$. Let $f$ be an element of the invariant control set $C$ of the full flag $\mathbb{F}$ with $\pi_k(f)=[p]$. Such element exists because $C_k= {\pi}_{k}(C)$. Let $\Theta(S)=\{r_1,\ldots,r_n\}$ be the flag type of $S$ and observe that $\pi_{\Theta(S)}^{-1}(\pi_{\Theta(S)}(f))$ is a subset of $C$, where  $\pi_{\Theta (S)}:\mathbb{F}\rightarrow\mathbb{F}_{\Theta (S)}$ . Therefore, $\pi_k(\pi_{\Theta(S)}^{-1}(\pi_{\Theta(S)}(f)))\subset C_k$. Since $k\notin\Theta(S)$ then $\pi_{\Theta(S)}(f)=\left(V_1\subset\cdots\subset V_n\right)$
	with $\text{dim}V_i=r_i$, $1\leq i\leq n$. We have the following cases:\newline
	\textbf{Case 1:} Assume that $r_1<k<r_n$. In this case, there exists $l\in\{1,\ldots,n-1\}$ such that the elements of $\pi_k(\pi_{\Theta(S)}^{-1}(\pi_{\Theta(S)}(f)))$ are the $k$-subspaces that contain $V_l$ and are contained in $V_{l+1}$. Let $\{v_1,\cdots,v_{r_l}\}$ be a basis of $V_l$, and complete it to an ordered basis $\{v_1,\cdots,v_{r_l},v_{r_l+1},\cdots,v_{r_{l+1}}\}$ of $V_{l+1}$. Since $r_l<k$ and $r_{l+1}>k$, consider the element $v_k$ in this basis of $V_{l+1}$ and, moreover, there is a basic element $v_{j}$ with $k<j\leq r_{l+1}$. In this way, define the subspace
	$$V=\{v_1\wedge\cdots\wedge v_{r_l}\wedge\cdots\wedge v_{k-1}\wedge(\alpha v_k+\beta v_j);\ \alpha,\beta\in\mathbb{R}\}.$$
	\newline
	\textbf{Case 2:} Now, suppose that $k<r_1$. Here, the elements of $\pi_k(\pi_{\Theta(S)}^{-1}(\pi_{\Theta(S)}(f)))$ are the $k$-subspaces contained in $V_{1}$. Since $k\geq1$, then $r_1\geq2$. Hence, given an ordered basis $\{v_1,\cdots,v_{r_1}\}$ of $V_1$, we can find $v_{k},v_{j}\in \{v_1,\cdots,v_{r_1}\}$ where $j$ satisfies $k<j\leq r_1$. Consider the subspace
	\begin{eqnarray}\label{setV}
V=\{v_1\wedge\cdots\wedge v_{k-1}\wedge(\alpha v_k+\beta v_j);\ \alpha,\beta\in\mathbb{R}\}.
	\end{eqnarray}
	\newline
	\textbf{Case 3:} Finally, assume $k>r_n$. Hence, $\pi_k(\pi_{\Theta(S)}^{-1}(\pi_{\Theta(S)}(f)))$ is the set formed by the $k$-subspaces which contains $V_{n}$. Since $k\leq d-1$, we can consider a basis $\{v_1,\ldots,v_{r_n}\}$ of $V_{r_n}$ and complete it to obtain the  ordered basis $\{v_1,\ldots,v_{r_n},v_{r_n+1},\ldots,v_d\}$ of $\mathbb{R}^d$. In this case, we can also take $v_k$ and $v_j$ in this basis, with $k<j\leq d$ and consider the subspace defined as in (\ref{setV}).
	
	In the three cases, the subspace $V\subset\bigwedge^k\mathbb{R}^d$ is two-dimensional and satisfies 
	$$\pi(V)\subset\varphi(\pi_k( \pi^{-1}_{\Theta(S)}(\pi_{\Theta(S)}(f))))\subset\varphi(C_k).$$
\end{proof}

The following theorem is a reciprocal of Theorem \ref{theo5}.

\begin{theorem}\label{teo10}
	If $\{0\}\neq W\subset\bigwedge^k\mathbb{R}^d$ is an  $S$-invariant cone, then $k\in\Theta(S)$.
\end{theorem}

\begin{proof}
	Assume that $k\notin\Theta(S)$ and denote by $L$ the intersection of $W$ with the set $\mathcal{D}$ of the decomposable elements of $\bigwedge^k\mathbb{R}^d$. By Proposition \ref{lemdec} we have that $L$ is nonempty. Moreover, $L$ is $S$-invariant, since the set of decomposable elements is also $S$-invariant. Therefore, $\varphi^{-1}(\pi(L))$ is also invariant. As $W$ is a closed set then $L$ is closed in $\mathcal{D}$ and hence $\varphi^{-1}(\pi(L))$ is a closed set in $\mathbb{G}_k(d)$. Since $\mathbb{G}_k(d)$ is compact, $\varphi^{-1}(\pi(L))$ is also compact, then there is an invariant control set contained in $\varphi^{-1}(\pi(L))$. But there is only one invariant control set  $C_{k} \subset \mathbb{G}_k(d)$ implying that $C_k\subset\varphi^{-1}(\pi(L))$ and hence $\pi^{-1}(\varphi(C_k))\subset L\subset W$. As proved in Lemma \ref{lema}, there is a two-dimensional subspace $V$ such that $\pi(V)\subset\varphi(C_k)$. But this means that $V\subset\pi^{-1}(\varphi(C_k))\subset W$, which is a contradiction because $W$ is pointed (see Proposition \ref{lem7}).
\end{proof}

 Recall that  if $S\subset \mathrm{Sl}(d,\mathbb{R})$ is a nonempty semigroup, then $S$ is transitive on 
	$\mathbb{R}^{d}\backslash \{0\}$ if and only if $S=\mathrm{Sl}(d,\mathbb{R})$ (see \cite{RSS}). In this context, the next theorem  gives a necessary and sufficient condition in terms of the existence of invariant cones.

\begin{theorem}\label{controllab}
	Let $S\subset {\rm Sl}(d,\mathbb{R})$ be a semigroup with nonempty interior. Then $S={\rm Sl}(d,\mathbb{R})$ if and only if there are no  $S$-invariant  cones in $\bigwedge^k\mathbb{R}^d$, for all $k\in\{1,\ldots,d-1\}$.
\end{theorem}

\begin{proof}
Let $W\subset\bigwedge^k\mathbb{R}^d$ be a proper $S$-invariant cone, for some $k\in\{1,\ldots,d-1\}$. Note that $W$ does not contain $\mathcal{D}$,  otherwise the convexity of $W$ would  imply that the convex closure of $\mathcal{D}$, $\bigwedge^k\mathbb{R}^d$, would be contained in $W$, which would contradicts the fact that $W$ is proper.

By Proposition \ref{lemdec} we can consider an element $v_1\in W\cap\mathcal{D}$. Take $v_2\in\mathcal{D}\backslash W$. If $S=\text{Sl}(d,\mathbb{R})$ and knowing that $\mathcal{D}$ is $S$-invariant then there exists $g\in S$ such that $gv_1=v_2\notin W$, but this contradicts the $S$-invariance of $W$. Hence $S\neq\text{Sl}(d,\mathbb{R})$.

On the other hand, assume that $S\subset \text{Sl}(d,\mathbb{R})$ is proper. Then $\Theta(S)\neq\emptyset$, hence there exists $k\in\Theta(S)$, for some $k\in\{1,\ldots,d-1\}$. Therefore, Theorem \ref{theo5} implies the existence of a such cone.
\end{proof}

\begin{remark}
	This theorem complement and improve Section 7 of \cite{RSS}.
\end{remark}

The next example shows that, as we commented before, the connectedness of $S$ is fundamental in the previous results.

\begin{example}
	
	Let $S^+\subset {\rm Sl}(2, \mathbb{R})$ be the set of matrices with positive entries. 
	It is not difficult to show that that  $S^+$ is a proper semigroup with nonempty interior in ${\rm Sl}(2, \mathbb{R})$,   the positive orthant 
	$Q^+=\{(a,b)\in\mathbb{R}^;a,b\ge 0\}$ 
	is  $S^+$-invariant and  $S^+$ is a  open set. Now take the following  proper semigroup
	$$S=S^+\cup(-S^+)=(-1)^{\mathbb{Z}}S^+=\{(-1)^kA;k\in\mathbb{Z},A\in S^+\}.$$
 Note that $S$  has nonempty interior. Moreover, $S$ is  not transitive on $\mathbb{R}^2$ because it leaves invariant the double cone $Q^+\cup -Q^+=(-1)^\mathbb{N}Q^+$:
	$$S((-1)^\mathbb{N}Q^+)=(-1)^\mathbb{N}S^+(-1)^\mathbb{N}Q^+=(-1)^{\mathbb{N}+\mathbb{N}}S^+Q^+=(-1)^{\mathbb{N}}Q^+.$$
	However, $S$ does not leave invariant proper cones in $\mathbb{R}^2=\bigwedge^1\mathbb{R}^2$. In fact, we have that $-I\in S$, therefore, if $C$ is a proper $S$ invariant cone then $-I(C)=-C\subset C$. This implies that $C$ is a subspace, which is a contradiction.
\end{example}

As a consequence of the above results, we get a necessary and sufficient condition for controllability of 
$$\dot{x}=Ax+uBx,x\in \mathbb{R}^{d}\setminus \{0\},u\in \mathbb{R}  ,$$
with $A,B\in\mathfrak{sl}(d,\mathbb{R})$.

Recall that the system semigroup  
$$S=\{e^{t_1(A+u_1B)}\cdots e^{t_n(A+u_nB)};\ t_1,\ldots,t_n\geq0,\ u_1,\ldots,u_n\in\mathbb{R},\ n\in\mathbb{N}\}$$
is a semigroup of ${\rm Sl}(d,\mathbb{R})$. Moreover, if the Lie algebra, generated by $A$ and $B$, coincides with $\mathfrak{sl}(d,\mathbb{R})$, then $\text{int}S\neq\emptyset$. Furthermore, $S$ is path connected. It is well know that this system is controllable if, and only if, $S={\rm Sl}(d,\mathbb{R})$ (see e.g. \cite{RSS}). Hence,  as a result of Theorem \ref{controllab} we have the necessary and sufficient condition for controllability of this bilinear system.

\begin{theorem}
The above system is controllable if and only if  it does not leave  invariant a cone in  $\bigwedge ^{k}{\Bbb R}^{d}$, for all $k \in \{1, \ldots , d-1 \}$.
\end{theorem}

\section{Flag type and invariance of convex sets}

In this section, we generalize the previous one. Or rather, instead of proper cones, we study the existence of proper convex sets in $\bigwedge^k\mathbb{R}$ which are invariant by the action of a semigroup $S \subset \text{Sl}(d,\mathbb{R})$. We also  relate the existence of this convex sets  with the flag type $\Theta(S)$ of $S$. 

Initially,  given $h\in \text{reg}(S)$, take as before the  basis $\mathcal{B} (h)=\{e_1,\ldots,e_d\}$ of $\mathbb{R}^d$.  Since $1=\det(h)=\lambda_1\cdots\lambda_d$, then for all $k\in\{1,\ldots,d-1\}$, we can prove that $\lambda_1\cdots\lambda_k>1$. 

The following lemma gives  an expression for the closed convex cone generated by a convex set in $\bigwedge^k\mathbb{R}^d$.

\begin{lemma}
	If the set $K\subset\bigwedge^k\mathbb{R}^d$ is convex, then the closed convex cone $W$ generated by $K$ is
	$$W:={\rm cl}(\displaystyle\bigcup_{\alpha>0}\alpha K).$$
\end{lemma}

\begin{proof}
	Let $\{W_l\}_{l\in\Lambda}$ be the family of all closed cones that contains $K$ and consider $V:=\displaystyle\bigcap_{l\in\Lambda}W_l$ the closed convex cone generated by $K$.
	
	Note that $W$ is a closed cone which contains $K$. To show that $W$ is convex, take $x,y\in W$. There are sequences $(\gamma_nx_n),(\delta_ny_n)$ in $\displaystyle\bigcup_{\alpha>0}\alpha K$ with $\gamma_n,\delta_n>0$ and $x_n,y_n\in K$ (for all $n\in\mathbb{N}$) converging to $x$ and $y$ respectively. Take $t\in[0,1]$ and define
	$$z_n=\left(\dfrac{(1-t)\gamma_n}{(1-t)\gamma_n+t\delta_n}\right)x_n+\left(\dfrac{t \delta_n}{(1-t)\gamma_n+t\delta_n}\right)y_n, \ n\in\mathbb{N}.$$
	Note that $(z_n)$ is a sequence in $K$, then $\left(\left((1-t)\gamma_n+t\delta_n\right)z_n\right)$ is a sequence in  $\bigcup\limits_{\alpha>0}\alpha K$, since $(1-t)\gamma_n+t\delta_n>0$. But $\left((1-t)\gamma_n+t\delta_n\right)z_n=(1-t)\gamma_n x_n+t\delta_n y_n$ converges to $(1-t)x+ty$,
	hence $(1-t)x+ty\in W$.	Therefore $V\subset W$.
	
	 On the other hand,  for each $\gamma>0$ we have $\gamma K\subset W_l$, for all $l\in\Lambda$, then $\displaystyle\bigcup_{\gamma>0}\gamma K\subset W_l$, for all $l\in\Lambda$. Hence the closeness of each $W_l$ implies that $W\subset W_l$, for all $l\in\Lambda$, so $W\subset V$.
\end{proof}

\begin{proposition}\label{pro13}
	Let $K\subset\bigwedge^k\mathbb{R}^d$ be a proper $S$-invariant convex set. Then the closed cone generated by $K$ is $S$-invariant.
\end{proposition}

\begin{proof}
	Denote by $W$ the closed cone generated by $K$. Since $K$ is $S$-invariant, for each $g\in S$ it holds that $gK\subset K$. Hence
	$$gW=g\left({\rm cl}({\displaystyle\bigcup_{\alpha>0}\alpha K})\right)\subset {\rm cl}({g\left(\displaystyle\bigcup_{\alpha>0}\alpha K\right)})={\rm cl}({\displaystyle\bigcup_{\alpha>0}\alpha gK})\subset {\rm cl}(\displaystyle\bigcup_{\alpha>0}\alpha K)=W,$$
	that is, $W$ is $S$-invariant.
\end{proof}

\begin{proposition}\label{prop13}
	If $K\subset\bigwedge^k\mathbb{R}^d$ is a proper $S$-invariant convex set, then $0\notin{\rm int}K$.
\end{proposition}

\begin{proof}
For each $h\in {\rm reg}(S)$ denote by $b_k(h)$ the attractor of $h$ in $G_k$. The set of transitivity of $C_k$,  $C_k^0$, satisfies
$$C_k^0=\{b_k(h);h\in {\rm reg}(S)\},$$
and has nonempty interior. In particular, there is an open set $V\subset C_k^0=\{b_k(h);h\in {\rm reg}(S)\}$. As a consequence, $\phi(V)$ is an open set in $\mathcal{D}$, and therefore, by Lemma \ref{theo4}, $\pi^{-1}(\phi(V))$ contains a basis $\{b_1,b_2, \ldots ,b_n\}$ of the exterior space. Since $b_i \in \pi^{-1}(\phi(V))$ and $V$ is a subset of $C_k^0=\{b_k(h);h\in  {\rm reg}(S)\}$, then, for each $b_i$ exists  $h_i\in {\rm reg}(S)$ such that $b_i\in\pi^{-1}(\phi(b_k(h_i)))$ or, equivalently, there is a basis $\{e_1(h_i),e_2(h_i), \ldots ,e_d(h_i)\}$ of $\mathbb{R}^d$ where $h_i$ is written as ${\rm diag}(\lambda_{1i},\lambda_{2i}, \ldots ,\lambda_{di})$ and $b_i=e_1(h_i)\wedge e_2(h_i)\wedge \cdots \wedge e_k(h_i)=e_I(h_i)$ with $I=\{1, \ldots ,k\}$.
 So, if we suppose that $0\in\text{int}K$, then there are $\alpha\neq0$ and $h_1,\ldots,h_r\in\text{reg}(S)$ with $r={{d}\choose{k}}$, such that $\alpha e_I(h_i)$ is a basis of $\bigwedge^k\mathbb{R}^d$ with $\pm\alpha e_I(h_i)$ contained in $\text{int}K$, $i=1,\ldots,r$.

But for all $m\in\mathbb{N}$ and $i\in\{1,\ldots,r\}$ we have $h_i^m(\pm\alpha e_I(h_i))\in K$ due to $S$-invariance of $K$. Moreover,
$$\|h_i^m(\pm\alpha e_I(h_i))\|=|\alpha|(\lambda_{1i}\cdots\lambda_{ki})^m\|e_I(h_i)\|\rightarrow+\infty,$$
then the convexity of $K$ implies that $K=\bigwedge^k\mathbb{R}^d$.
\end{proof}

The above proposition has the following consequence.

\begin{corollary}\label{corol15}
Let $K\subset\bigwedge^k\mathbb{R}^d$ be an $S$-invariant convex set and denote by $W$ the closed cone generated by $K$. The following statements are equivalents: 
\begin{itemize}
	\item[i)] $W$ is proper;
	\item[ii)] $K$ is proper.
	\item[iii)] $0\notin {\rm int}K$.
\end{itemize}
\end{corollary}

\begin{proof}
	The implication $(i)\Rightarrow(ii)$ holds because $K\subset W$. Moreover, $(ii)\Rightarrow(iii)$ follows by Proposition \ref{prop13}. Finally, to prove that  $(iii)\Rightarrow(i)$ we first note that if  $W=\bigwedge^k\mathbb{R}^d$ then  $\text{int}K\neq\emptyset$. In fact, if  $\text{int}K= \emptyset$ then  $K$ is contained in a proper affine subspace $V+u_0$, where $V\subset\bigwedge^k\mathbb{R}^d$ is a proper vector subspace and $u_0\in\bigwedge^k\mathbb{R}^d$. Hence
	\begin{eqnarray*}
		\bigwedge^k\mathbb{R}^d&=&W={\rm cl}(\displaystyle\bigcup_{\alpha>0}\alpha K)\subset{\rm cl}(\displaystyle\bigcup_{\alpha>0}\alpha(V+u_0))={\rm cl}(\displaystyle\bigcup_{\alpha>0}(V+\alpha u_0))\\
		&=&V+[0,+\infty)u_0\not\subseteq\bigwedge^k\mathbb{R}^d.
	\end{eqnarray*}
	which is a contradiction. Hence, given the open set $-\text{int}K$, there are $\alpha>0$ and $k\in K$ with $\alpha k\in-\text{int}K$, that is, $-\alpha k\in\text{int}K$. Since $K$ is convex, the line $[-\alpha k,k):=\{(t-1)\alpha k+tk;\ t\in[0,1)\}$ is contained in $\text{int}K$, therefore $0\in\text{int}K$.
\end{proof}

The next result presents a synthesis of this section, the relation among invariant convex set, invariant cone and flag type.

\begin{theorem}
	Let $S\subset {\rm Sl}(d,\mathbb{R})$ a semigroup with nonempty interior. Then the following statements are equivalents:
	\begin{itemize}
		\item[i)] There exists an $S$-invariant	proper convex set in $\bigwedge^k\mathbb{R}^d$;
		\item[ii)] There exists an $S$-invariant proper closed cone in $\bigwedge^k\mathbb{R}^d$;
		\item[iii)] $k\in\Theta(S)$.
	\end{itemize}
\end{theorem}

\begin{proof}
	By Proposition \ref{pro13} and Corollary \ref{corol15} we have that  $(i)\Rightarrow(ii)$. By Theorem \ref{teo10} it follows that $(ii)$ implies $(iii)$. Moreover, since a cone is a convex set, the implication $(iii)\Rightarrow(i)$ follows by Theorem \ref{theo5}.
\end{proof}

\section{Examples}

In order to present  examples to illustrate our results, we create a computational implementation in Julia Language \cite{BEK} called \texttt{LieAlgebraRankCondition.jl}\footnote{Available in \url{https://github.com/evcastelani/LieAlgebraRankCondition.jl}}. The basic idea of this implementation is the following: given the bilinear control system $$\dot{x}=Ax+uBx,x\in \mathbb{R}^{4}\setminus \{0\},u\in \mathbb{R} \mbox{ and } A,B \in \mathfrak{sl}(4,{\Bbb R})$$ put the Lie brackets in a convenient way and analyse all the possibilities until get, if possible, a linearly independent (L.I.) set for $\mathfrak{sl}(4,{\Bbb R})$. In the following we describe a conceptual algorithm.

\begin{algorithm}[H]
	\KwData{$A$: Array, $B$: Array, \texttt{dim}: dimension of $\mathfrak{sl}(4,{\Bbb R})$}
	\KwResult{\texttt{True}: a set of L. I. arrays were found; \texttt{False}: Does not exists an L. I. set of arrays.}
	$C \leftarrow \{A,B, [A,B]\}$\;
	\eIf{ $C$ is L. I.}
	{ $k\leftarrow 3$\;}
	{return \texttt{False}\;} 
	\While{$k\leq$ \texttt{dim}}{
		$j \leftarrow k-1$\;
		$C_{trial} \leftarrow C_j$\;
		\While{($C\cup [C_{trial},C_k]$ is not L.I) and ($j>3$)}{
			$j \leftarrow j-1$ \;
			$C_{trial}\leftarrow C_j$\;
		}
		\eIf{j=3}{
			remove $C_k$ from $C$ \;
			$k \leftarrow k-1$\;
		}
		{	
			add $[C_{trial},C_k]$ to $C$\;
			$k\leftarrow k+1$\;
		}
		\If {k=3}{
			return \texttt{False}\;
		}
	}
	return \texttt{True}\;
	\caption{Lie Algebra Rank Condition Algorithm.}
	\label{alg:alg1}
\end{algorithm}

\begin{remark}
	The parameter \texttt{dim} can be changed in order to find solutions for higher order spaces. 
\end{remark}

\begin{example}
	Consider the bilinear system
	$$ (\Sigma) \,\,\,\, \dot{x}=Ax+uBx, \mbox{ with } x\in \mathbb{R}^{4}\setminus \{0\} ,  u\in \mathbb{R} , $$
		$$A= \left[
	\begin{array}{cccc}
	0 & 2 & 0 & -1\\
	2 & 0 & 2 & 0\\
	0 & 2 & 0 & 2\\
	-1 & 0 & 2 & 0\\
	\end{array}
	\right]\mbox{ and } B={\rm diag}(4,1,-2,-3) \in \mathfrak{sl}(4,\mathbb{R}). $$

	The matrix $A$ has the distinct eigenvalues, $3,2,-2,-3$, with the following eigenvectors $v_1=(1,2,2,1)$, 
	$v_2=(-2,-1,1,2)$, $v_3=(2,-1,-1,2)$ and $v_4=(-1,2,-2,1)$, respectively.
	Let $S$ be the semigroup of $(\Sigma)$, that is,
	$$S=\{e^{t_1(A+u_1B)}\cdots e^{t_k(A+u_nB)};t_1,\ldots,t_n \geq0,\ n\in\mathbb{N}\}.$$
	Using the implementation of Algorithm \ref{alg:alg1}, we can show that this system satisfies the Lie algebra rank condition, hence $S$ has nonempty interior in  ${\rm Sl}(4, {\Bbb R})$. Moreover, $S$ is a proper semigroup. In fact, by 
	\cite[Proposition 2]{SM2}, we have
	$$A+uB\in\mathcal{L}(S_2)=\{X\in\mathfrak{sl}(4,\mathbb{R});\ \exp(X)\in S_2\},$$
	where $S_2=\{g\in {\rm Sl}(4,\mathbb{R});\ g\mathcal{O}_2\subset\mathcal{O}_2\}$ is the the compression semigroup of the positive orthant $\mathcal{O}_2=\left\{\displaystyle\sum_{I=\{i_1<i_2\}\subset\{1,2,3,4\}}\alpha_Ie_I;\ \alpha_I\geq0\right\}\subset\bigwedge^2\mathbb{R}^4$. 
	This semigroup coincides with the set of all matrix in ${\rm Sl}(4,\mathbb{R})$ such that the minors of order $2$ have non-negative determinant. Note that  $S\subset S_2$. Since $S_2$ leaves invariant the cone $\mathcal{O}_2$, then $S \mathcal{O}_2\subset\mathcal{O}_2$. Hence $(\Sigma)$ is not controllable and therefore $S$ is proper, in particular $S$ leaves invariant the positive orthant of  $\bigwedge^2\mathbb{R}^4$.
	
	On the other hand,  neither $\pm A$ nor $\pm(A+uB)$ leave invariant an orthant of $\mathbb{R}^4$.	In fact, by \cite[Lemma 1]{YSac}, a matrix $X=(x_{ij})$ leaves invariant the orthant with signs $(\sigma_1,\ldots,\sigma_d)$ if and only if $\sigma_i\sigma_jx_{ij}>0$. Applying this condition	to $\pm A$, $\pm(A+uB)$, we get  the contradictory fact that $\sigma_1\sigma_4$ must be simultaneously $1$ and $-1$, so that there are no invariant orthants in $\mathbb{R}^4 =\bigwedge^1\mathbb{R}^4$. The  system $(\Sigma)$ is a counter-example for the following conjecture proposed by  Sachkov in \cite{YSac}. 	
	Consider  a bilinear control system with $A$ symmetric and $B={\rm diag}(b_1,\ldots,b_n)$ where $b_i\neq b_j$ for $i\neq j$. Is it true that if this system has no invariant orthants and everywhere satisfies the necessary Lie algebra rank controllability condition, then it is controllable in $\mathbb{R}^d\setminus\{0\}$?
	
	Now we prove that, although  $(\Sigma)$ is not controllable, there are no  $S$-invariant cones in $\mathbb{R}^4 =\bigwedge^1\mathbb{R}^4$ neither in $\bigwedge^3\mathbb{R}^4$.
	Suppose that $W\subset\mathbb{R}^4$ is an $S$-invariant cone. Then $W$ has nonempty interior and it is not contained in the plane generated by $\{e_2,e_3,e_4\}$. Therefore there is a vector $w=(w_1,w_2,w_3,w_4)\in W$ such that $w_1\neq 0$. Since $e^{tB} \in {\rm cl(S)}$  for all $t \in \mathbb{R}$, then if $w_1>0$ we have that
	$$\lim_{t\rightarrow+\infty}\frac{e^{tB}w}{\Vert e^{tB}w\Vert}=e_1\in W.$$
	If $w_1<0$ then
	$$\lim_{t\rightarrow+\infty}\frac{e^{tB}w}{\Vert e^{tB}w\Vert}=-e_1\in W.$$
	Without loss of generality, assume that $e_1\in W$. Knowing that $v_1$ is the attractor eigenvalue of $A$ and considering the  basis $\{v_1,v_2,v_3,v_4\}$, a similar argument  assures that either $v_1\in V$ or $-v_1\in V$.
	
	Let $H:=\{(x_1,x_2,x_3,x_4)\in\mathbb{R}^4:x_4<0\}$, then, for all $x\in H$,
	$$\lim_{t\rightarrow+\infty}\dfrac{e^{t(-B)}x}{||e^{t(-B)}x||}=-e_4.$$
	
	In particular, note that if $W\cap H\neq\emptyset$, then $-e_4\in W$. Now we show  that $W\cap H\neq\emptyset$. Since the inner product between $Ae_1$ and $e_4$ is negative, then the curve $t\mapsto e^{tA}e_1$ intersects $H$ for  $t>0$. By $S$-invariance and knowing that $e_1\in W$, we have $e^{\mathbb{R}_+A}e_1\subset W$ , then $W\cap H\neq\emptyset$. 
	
	As stated early, either $v_1\in W$ or $-v_1\in W$. As $v_1$ has a positive fourth coordinate, then
	$$\lim_{t\rightarrow+\infty}\dfrac{e^{t(-B)}v_1}{||e^{t(-B)}v_1||}=e_4,$$
	and as $-v_1$ has a negative first coordinate, we have
	$$\lim_{t\rightarrow+\infty}\dfrac{e^{tB}(-v_1)}{||e^{tB}(-v_1)||}=-e_1.$$
	
	Hence if $v_1\in W$ then $e_4\in W$. But $-e_4$ is also in $W$, then $W$ is not pointed. On the other hand, if $-v_1\in W$, then $-e_1\in W$. Analogously, since $e_1$ is also in $W$, then $W$ is not pointed also in his case. Anyway $W$ is not pointed, but this contradicts Proposition \ref{lem7}.
	
	Since $W$ is arbitrary, we conclude that $(\Sigma)$ does not have   invariant cones in $\mathbb{R}^4 =\bigwedge^1\mathbb{R}$.
	
	Now in the case of $\bigwedge^3\mathbb{R}^4$, we  recall that  $S$ has  invariant cones in $\bigwedge^3\mathbb{R}^4$ if, and only if, $S^{-1}$ has  invariant cones in $\mathbb{R}^4$, and the linear isomorphism from $\mathbb{R}^4$ to $\bigwedge^3\mathbb{R}^4$ (that preserves basis) is also a one to one correspondence between the respective invariant cones  (see e.g. \cite{SM2}).
	
	Therefore, it is enough to prove that $S^{-1}$ does not leave  invariant cones in $\mathbb{R}^4$. Since $S$ is generated by the exponential of the elements of $\{A+uB;\ u\in\mathbb{R}\}\subset\mathfrak{sl}(4,\mathbb{R})$, then $S^{-1}$ is generated by the exponential of the elements $-A+uB$ with $u\in\mathbb{R}\}$. Then $S^{-1}$ is also the semigroup of the bilinear control system $\dot{x}=-Ax+uBx$ with $x\in \mathbb{R}^{4}\setminus \{0\}$ and $u\in \mathbb{R}$.
	
	 Let $W\neq\{0\}$ be an $S^{-1}$-invariant cone. Note that $S^{-1}$ has nonempty interior in ${\rm Sl}(4,\mathbb{R})$. Therefore,  $e_1\in W$ or $-e_1\in W$. Without loss of generality, we assume $e_1\in W$. Since the highest eigenvalue of $-A$ is $3$ and the corresponding eigenvector is $v_4$, then  $v_4\in W$ or $-v_4\in W$. Furthermore, the inner product between $-Ae_1$ and $e_4$ is positive, and, therefore, $e_4\in W$. If $v_4\in W$, then
	$\displaystyle\lim_{t\rightarrow +\infty}\dfrac{e^{tB}v_4}{||e^{tB}v_4||}=-e_1\in W$ and
	$W$ is not pointed, because $e_1,-e_1\in W$. Otherwise, if $-v_4\in W$, then
	$\displaystyle\lim_{t\rightarrow+\infty}\dfrac{e^{t(-B)}(-v_4)}{|| e^{t(-B)}(-v_4)||}=-e_4$
	and $W$ is still not pointed, because $e_4,-e_4\in W$.
	Since $W$ is not pointed in both cases, by Proposition \ref{lem7} we have a contradiction. We conclude that  the proper semigroup $S$ does not leave  invariant a proper cone in $\bigwedge^1\mathbb{R}^4$ neither in $\bigwedge^3\mathbb{R}^4$ but  $S$ has an invariant cone in $\bigwedge^2\mathbb{R}^4$ (in fact, we showed that it leaves invariant the positive orthant of that space). Then by Theorem \ref{controllab}, the system $(\Sigma)$ is not controllable. Moreover, Theorem \ref{teo10} implies that $S$ has parabolic type $\Theta(S)=\{2\}$, in other words, $\mathbb{F}_{\Theta(S)}=\mathbb{G}_2(4)$.
	
\end{example}

\begin{example}
	
	\noindent Consider the above bilinear control system, but with 
	$$A:= \left[
	\begin{array}{cccc}
	1&1&0&0\\
	-1&1&0&0\\
	0&0&-1&\frac{1}{2}\\
	0&0&-\frac{1}{2}&-1
	\end{array}
	\right],\ B:=\left[
	\begin{array}{cccc}
	2&0&0&0\\
	0&-\frac{3}{2}&-\frac{1}{10}&0\\
	0&\frac{1}{10}&-\frac{3}{2}&0\\
	0&0&0&1
	\end{array}
	\right]$$
	and denote the system semigroup by $S$. Using again the implementation of Algorithm \ref{alg:alg1}, we can see that $S$ satisfies the Lie algebra rank condition, so ${\rm int}S\neq\emptyset$. Now we show that $S$ does not have invariant cones in ${\bigwedge^{1}\mathbb{R}^4},{\bigwedge^{2}\mathbb{R}^4}$ or ${\bigwedge^{3}\mathbb{R}^4}$ and therefore  $S = {\rm Sl}(4,{\Bbb R})$.\\
	First note that 
	$$e^{\frac{\pi}{2}A}=\left[
	\begin{array}{cccc}
	0&d&0&0\\
	-d&0&0&0\\
	0&0&\frac{1}{d}\frac{\sqrt{2}}{2}&\frac{1}{d}\frac{\sqrt{2}}{2}\\
	0&0&-\frac{1}{d}\frac{\sqrt{2}}{2}&\frac{1}{d}\frac{\sqrt{2}}{2}
	\end{array}
	\right]$$
	with $e^{\frac{\pi}{2}}=d$.
	
	Now we compute $e^{\frac{\pi}{2}A}$ in the canonical basis of ${\bigwedge^{3}\mathbb{R}^4}$. 
	$$e^{\frac{\pi}{2}A}(e_1\wedge e_2\wedge e_3)=d\frac{\sqrt2}{2}e_1\wedge e_2\wedge e_3-d\frac{\sqrt 2}{2}e_1\wedge e_2\wedge e_4 , $$
	$$e^{\frac{\pi}{2}A}(e_1\wedge e_2\wedge e_4)=d\frac{\sqrt2}{2}e_1\wedge e_2\wedge e_3+d\frac{\sqrt2}{2}e_1\wedge e_2\wedge e_4 , $$
	$$e^{\frac{\pi}{2}A}(e_1\wedge e_3\wedge e_4)=-\frac{1}{d}e_2\wedge e_3\wedge e_4$$
	and
	$$e^{\frac{\pi}{2}A}(e_2\wedge e_3\wedge e_4)=\frac{1}{d}e_1\wedge e_3\wedge e_4 . $$
	Then $e^{\frac{\pi}{2}A}$ can be written, with respect to the canonical basis of ${\bigwedge^{3}\mathbb{R}^4}$, as
	$$\left[
	\begin{array}{cccc}
	d\frac{\sqrt{2}}{2}&d\frac{\sqrt{2}}{2}&0&0\\
	-d\frac{\sqrt{2}}{2}&d\frac{\sqrt{2}}{2}&0&0\\
	0&0&0&\frac{1}{d}\\
	0&0&-\frac{1}{d}&0
	\end{array}
	\right]= \left[
	\begin{array}{cc}
	d I&0\\
	0&\frac{1}{d}I
	\end{array}
	\right]\left[
	\begin{array}{cc}
	R_1&0\\
	0&R_2
	\end{array}
	\right]
	$$
	with $R_1,R_2$ rotations by angles different from $0$ and $\pi$.
	
	In the next lemma we prove that the   cones in  $\mathbb{R}^4$, which are invariant by above matrix, are subspaces. 
	\begin{lemma}
		Let $T\in {\rm Sl}(4,{\Bbb R})$ be the matrix 
		$$
		T=\left[
		\begin{array}{cc}
		d I & 0\\
		0 & \frac{1}{d}I
		\end{array}
		\right]\left[
		\begin{array}{cc}
		R_1 & 0\\
		0 & R_2
		\end{array}
		\right]$$
		where $R_1 , R_2 \in {\rm SO}(2,{\Bbb R}) \setminus \{I,-I\}$, $I$ is $(2 \times 2)$-identity matrix and $d \in \mathbb{R} \setminus \{0\}$. If $W$ is a $T$-invariant cone in $\mathbb{R}^4$   then $W$ is a subspace.
	\end{lemma}
	\begin{proof}
		Note that  $\langle e_1 , e_2\rangle$ and $\langle e_3 , e_4 \rangle$ are $T$-invariant spaces, and the restrictions of $T$ to these spaces are  $\alpha R$ where $\alpha>0$ and $R$ is the rotation different from $I$ and $-I$. The only  cones in a two-dimensional space that are invariant by these maps are $(0,0)$ or the whole space, hence if  $W\subset\langle e_1 , e_2\rangle$ then $W=\{0\}$ or $W=\langle e_1 , e_2 \rangle$. If $W\subset \langle e_3,e_4\rangle$ then $W=\{0\}$ or $W=\langle e_3 , e_4 \rangle$. Suppose that $W$ is not contained in these spaces. Then there exists  $v\in W$ such that $v \neq \langle e_1 , e_2 \rangle$ and $v \neq \langle e_3 , e_4 \rangle$. As  $\mathbb{R}^4$ is a direct sum of these two spaces, then $v$ has the unique decomposition  $v=u+w$, with $0 \neq u\in\langle e_1 , e_2 \rangle$ and $0 \neq w \in\langle e_3 , e_4 \rangle$. Knowing the eigenvalues of the restriction of $A$ to $\langle e_1 , e_2 \rangle$ we can show that 
		$$\Vert T^nu\Vert=\Vert dR_1^nu\Vert\rightarrow+\infty \mbox{ and }
		\Vert T^nw|=\Vert(1/d)R_2^nw\Vert\rightarrow0 . $$
		In particular, the distance of $\frac{T^nv}{\Vert T^nv\Vert}$ to $\langle e_1 , e_2 \rangle$ converges to zero, this sequence is contained in a compact set and has a  subsequence  that converges to $p$. Note that  $p\in\langle e_1 , e_2 \rangle$ and  $\Vert p\Vert =1$. As $W$ is a $T$-invariant cone then $p\in {\rm cl}(W)=W$. 
		
		We have also that $W\cap\langle e_1 , e_2 \rangle$ is a $T$-invariant  cone which contains $p$. Then  $W\cap\langle e_1 , e_2\rangle =\langle e_1 , e_2 \rangle$ and so $\langle e_1 , e_2 \rangle\subset W$. It implies that  $-u\in W$, then $w=v+(-u)\in W$ and therefore   $W$ has a non-null element of $\langle e_3 , e_4 \rangle$. In a similar way we can see that  $\langle e_3 , e_4\rangle\subset W$. Hence  $W$ contains  $\langle e_1 , e_2 \rangle$ and $\langle e_3 , e_4 \rangle$, that is, $W=\mathbb{R}^4$. In all cases, $W$ is a subspace of $\mathbb{R}^4$.
	\end{proof}
	
	By the above lemma, any $e^{\frac{\pi}{2}A}$-invariant cone in ${\bigwedge^{1}\mathbb{R}^4}$ or  in ${\bigwedge^{3}\mathbb{R}^4}$, is a subspace. Therefore there are no   $S$-invariant cones in ${\bigwedge^{1}\mathbb{R}^4}$ neither in ${\bigwedge^{3}\mathbb{R}^4}$.

	Now it remains to prove that in ${\bigwedge^{2}\mathbb{R}^4}$ there are no $S$-invariant cones. First note that 
	the following submatrix of $B$,
	$$B_2=\left[
	\begin{array}{cc}
	-\frac{3}{2}&-\frac{1}{10}\\
	\frac{1}{10}&-\frac{3}{2}
	\end{array}
	\right]$$
	satisfies $\displaystyle\lim_{t\rightarrow+\infty}e^{tB_2}=0$ implying that $\displaystyle\lim_{t\rightarrow+\infty}e^{tB}v=0$
	for all $v\in\langle e_2 , e_3 \rangle$. Moreover 
	$e^{tB}e_1=e^{2t}e_1$ and $e^{tB}e_4=e^te_4$.
	
	Note that when  $t\rightarrow+\infty$ we have that 
	$$\frac{e^{tB}(e_1\wedge e_4)}{e^{2t}e^{t}}=\frac{e^{2t} e_1\wedge e^{t}e_4}{e^{2t}e^{t}}=e_1\wedge e_4\rightarrow e_1\wedge e_4$$ and moreover 
	$$\frac{e^{tB}(e_i\wedge e_j)}{e^{2t}e^{t}} \rightarrow 0 \mbox{ for } (i,j)\neq (1,4) . $$
	
	Hence, for any vector  $v \in \bigwedge^2\mathbb{R}^4$ we have
	\begin{eqnarray}\label{vector}
v=\alpha_1 e_1\wedge e_4+\alpha_2 e_1\wedge e_2+\alpha_3 e_1\wedge e_3+\alpha_4 e_4\wedge e_2+\alpha_5 e_4\wedge e_3+\alpha_6 e_2\wedge e_3,
	\end{eqnarray}
	for some $v_{1}, \ldots , v_{4} \in \mathbb{R}$ and
	we have $\displaystyle\lim_{t\rightarrow+\infty}\frac{e^{tB}(v)}{e^{2t}e^{t}}= \alpha_1 e_1\wedge e_4$. Now, suppose that exists an $S$-invariant cone $W$. Then, there is $v\in W$ of the form (\ref{vector}) such that
	$$\lim_{t\rightarrow +\infty}\frac{e^{tB}(v)}{e^{2t}e^{t}}=\alpha e_1\wedge e_4$$
	with $\alpha\neq 0$, because ${\rm int}W\neq\emptyset$.
		
	As  $\alpha e_1\wedge e_4\in W$ and
	$$e^{2\pi A}=\left[
	\begin{array}{cccc}
	d^4&0&0&0\\
	0&d^4&0&0\\
	0&0&-\frac{1}{d^4}&0\\
	0&0&0&-\frac{1}{d^4}
	\end{array}
	\right]$$
	we have that $e^{2\pi A}(\alpha e_1\wedge e_4)=\alpha e_1\wedge -e_4=-\alpha e_1\wedge e_4.$
	As $W$ is invariant by the  $e^{2\pi A}$-action, then $-\alpha e_1\wedge e_4\in W$, hence any straight line generated by  $\alpha e_1\wedge e_4$ is contained in  $W$, that is, $W$ is not pointed. Consequently, ${\bigwedge^{2}\mathbb{R}^4}$ does not have   $S$-invariant cones. Therefore,  by Theorem \ref{controllab},  $S={\rm Sl}(4,{\Bbb R})$, that is, the system is controllable.
\end{example}

{\bf Acknowledgments:} The authors are greatly indebted to Prof. L.A.B. San Martin for suggesting the problem and for many stimulating conversations.

\end{document}